\documentclass[preprint,12pt]{elsarticle}

\biboptions{numbers,sort&compress}
\usepackage{amssymb}
\usepackage{amsthm}
\usepackage{amsmath}
\usepackage{epic}
\usepackage{CJK}
\usepackage{setspace}%%行间距
\usepackage{bm}
\newtheorem{thm}{Theorem}[section]
\newtheorem{cor}[thm]{Corollary}
\newtheorem{lem}[thm]{Lemma}

\newtheorem{example}[thm]{Example}

\newtheorem{obs}[thm]{Observation}
\journal{}

\begin{document}
\begin{spacing}{1.15}
\begin{CJK*}{GBK}{song}
\begin{frontmatter}
\title{\textbf{Oriented spanning trees and stationary distribution of digraphs}}

\author{Jiang Zhou}\ead{zhoujiang@hrbeu.edu.cn}
\author{Changjiang Bu}

\address{College of Mathematical Sciences, Harbin Engineering University, Harbin 150001, PR China}

\begin{abstract}
By using biclique partitions of digraphs, this paper gives reduction formulas for the number of oriented spanning trees, stationary distribution vector and Kemeny's constant of digraphs. As applications, we give a method for enumerating spanning trees of undirected graphs by vertex degrees and biclique partitions. The biclique partition formula also extends the results of Knuth and Levine from line digraphs to general digraphs.
\end{abstract}

\begin{keyword}
Oriented spanning tree, Matrix-tree theorem, Random walk, Stationary distribution
\\
\emph{AMS classification (2020):} 05C30, 05C50, 05C05, 05C81, 05C20
\end{keyword}
\end{frontmatter}

\section{Introduction}
Let $G$ be a weighted digraph with vertex set $V(G)$ and edge set $E(G)$, and each edge $e=ij\in E(G)$ is weighted by an indeterminate $w_e(G)=w_{ij}(G)$. The notation $e=ij\in E(G)$ means there exists a directed edge from tail vertex $i$ to head vertex $j$, and the tail and head of $e$ are denoted by $i=t(e)$ and $j=h(e)$, respectively. The outdegree and the indegree of a vertex $v$ in $G$ are denoted by $d_v^+(G)$ and $d_v^-(G)$, respectively. The word ``weighted" is omitted when $w_e(G)=1$ for each $e\in E(G)$. The \textit{line digraph} $\mathcal{L}(G)$ of a digraph $G$ has vertex set $V(\mathcal{L}(G))=E(G)$, and there exits directed edge from $e$ to $f$ in $\mathcal{L}(G)$ if and only if $h(e)=t(f)$.

An \textit{oriented spanning tree} \cite{Chung,Levine} of weighted digraph $G$ is a subtree containing all vertices of $G$, in which one vertex, the root, has outdegree $0$, and every other vertex has outdegree $1$. Let $\mathbb{T}_u(G)$ denote the set of oriented spanning trees of $G$ with root $u$, and let $\kappa_u(G)=|\mathbb{T}_u(G)|$. A spanning tree enumerator of $G$ is defined as
\begin{eqnarray*}
t_u(G)=\sum_{T\in\mathbb{T}_u(G)}\prod_{e\in E(T)}w_e(G).
\end{eqnarray*}
If $w_e(G)=1$ for each $e\in E(G)$, then $t_u(G)=\kappa_u(G)$ is the number of oriented spanning trees of $G$ with root $u$. Some algebraic formulas for $\kappa_u(G)$ and $t_u(G)$ were given in \cite{Chaiken,Chung,Tutte}.

In the study of random walks on digraphs, one basic problem is to determine the stationary distribution vector, which represents the long-term behaviour of the Markov chain associated with the digraph \cite{Aksoy}. The stationary distribution has applications in PageRank algorithms \cite{Andersen}.

There exists a closed relation between the stationary distribution vector and spanning tree enumerators of digraphs. For a strongly connected digraph $G$ with $n$ vertices, its stationary distribution vector $\pi=(\pi_1,\ldots,\pi_n)^\top$ satisfies \cite{Aldous,Kirkland} $\pi_i=\frac{t_i(\widetilde{G})}{\sum_{j=1}^nt_j(\widetilde{G})}$, where $\widetilde{G}$ is the weighted digraph obtained from $G$ by taking $w_{ij}(\widetilde{G})=d_i^+(G)^{-1}$ for each $ij\in E(G)$. Let $m_{ij}$ be the mean first passage time from $i$ to $j$, then the value $\mathcal{K}(G)=\sum_{j\in V(G),j\neq i}m_{ij}\pi_j$ (does not depend on $i$) is called the \textit{Kemeny's constant} \cite{Kirkland} of $G$.

Some formulas for counting spanning trees in undirected line graphs can be found in \cite{Dong,Gong,Yan,Zhou}. For line digraphs, Knuth \cite{Knuth} proved the following formula for counting oriented spanning trees. Combinatorial and bijective proofs of the Knuth formula are given in \cite{Orlin} and \cite{Bidkhori}, respectively.
\begin{thm}\label{thm1.1}
\textup{[13, Formulas (5) and (7)]} Let $G$ be a digraph such that $d_u^+(G)d_u^-(G)>0$ for each $u\in V(G)$. For any edge $e=ij$ of $G$, we have
\begin{eqnarray*}
\kappa_e(\mathcal{L}(G))=\left(\kappa_j(G)-d_j^+(G)^{-1}\sum_{kj\in E(G)\atop k\neq i}\kappa_k(G)\right)\prod_{v\in V(G)}d_v^+(G)^{d_v^-(G)-1}.
\end{eqnarray*}
If $d_u^+(G)=d_u^-(G)$ for each $u\in V(G)$, then
\begin{eqnarray*}
\kappa_e(\mathcal{L}(G))=d_j^+(G)^{-1}\kappa_u(G)\prod_{v\in V(G)}d_v^+(G)^{d_v^+(G)-1}~(u\in V(G)).
\end{eqnarray*}
\end{thm}
Let $\mathcal{E}(G)$ denote the number of Eulerian circuits of a digraph $G$. By Lemma \ref{lem2.2}, we can derive the following formula for counting Eulerian circuits of a line digraph from Theorem \ref{thm1.1}.
\begin{thm}\label{thm1.2}
\textup{[13, page 313]} Let $G$ be a digraph with $n$ vertices such that $d_u^+(G)=d_u^-(G)=d>0$ for each $u\in V(G)$. For any $e\in E(G)$ and $v\in V(G)$, we have
\begin{eqnarray*}
\kappa_e(\mathcal{L}(G))&=&d^{n(d-1)-1}\kappa_v(G),\\
\mathcal{E}(\mathcal{L}(G))&=&d^{-1}(d!)^{n(d-1)}((d-1)!)^{n}\kappa_v(G)=d^{-1}(d!)^{n(d-1)}\mathcal{E}(G).
\end{eqnarray*}
\end{thm}

Let $\{w_i(G)\}_{i\in V(G)}$ be indeterminates on $V(G)$. If each edge $\{u,v\}$ in $G$ has weight $w_{uv}(G)=w_v(G)$, then we say that \textit{the weights of $G$ are induced by} $\{w_i(G)\}_{i\in V(G)}$.

Levine \cite{Levine} proved the following formula for spanning tree enumerators of a line digraph, which is a generalization of Theorem \ref{thm1.1}.
\begin{thm}\label{thm1.3}
\textup{[14, Theorem 2.3]} Let $G$ be a weighted digraph such that $d_u^-(G)>0$ for each $u\in V(G)$. For any edge $e=ij$ of $G$ satisfying $d_j^-(G)\geq2$, we have
\begin{eqnarray*}
t_e(\mathcal{L}(G))=w_e(G)t_i(G)d_j(G)^{d_j^-(G)-2}\prod_{v\in V(G)\atop v\neq j}d_v(G)^{d_v^-(G)-1},
\end{eqnarray*}
where the weights of $\mathcal{L}(G)$ are induced by indeterminates $\{w_e(G)\}_{e\in V(G)}$, $d_v(G)=\sum_{vu\in E(G)}w_{vu}(G)$.
\end{thm}
A \textit{biclique} \cite{Gregory} is a bipartite digraph $Q$ whose vertices can be partitioned into two parts $Q^{(1)}$ and $Q^{(2)}$, and $E(Q)=\{ij:i\in Q^{(1)},j\in Q^{(2)}\}$. A \textit{biclique partition} of a digraph $G$ is a set $\varepsilon=\{Q_1,\ldots,Q_r\}$ of bicliques in $G$ such that each edge of $G$ belongs to exactly one biclique of $\varepsilon$. For a biclique partition $\varepsilon=\{Q_1,\ldots,Q_r\}$ of $G$, its \textit{biclique digraph} has vertex set $\varepsilon=\{Q_1,\ldots,Q_r\}$ and edge set $\{Q_iQ_j:Q_i^{(2)}\cap Q_j^{(1)}\neq\emptyset\}$.
\begin{obs}\label{obs1.4}
For a vertex $i$ of a digraph $G$, the edge sets $Q_i^{(1)}=\{e:e\in E(G),h(e)=i\}$ and $Q_i^{(2)}=\{f:f\in E(G),t(f)=i\}$ form a biclique in the line digraph $\mathcal{L}(G)$, and all such bicliques form a natural biclique partition $\varepsilon=\{Q_i:i\in V(G),d_i^+(G)d_i^-(G)>0\}$ of $\mathcal{L}(G)$. Notice that $Q_i^{(2)}\cap Q_j^{(1)}\neq\emptyset$ if and only if $ij\in E(G)$. So the biclique digraph of $\varepsilon$ is isomorphic to $G$ when $d_u^+(G)d_u^-(G)>0$ for each $u\in V(G)$.
\end{obs}
It is known that every digraph $G$ has a biclique partition $\varepsilon$ satisfying $|\varepsilon|\leq|V(G)|$ ($|\varepsilon|$ is much smaller than $|V(G)|$ for many digraphs). So we can get reduction formulas for spanning tree enumerators, stationary distribution vector and Kemeny's constant of $G$ by counting oriented spanning trees in the biclique digraph of $\varepsilon$. Moreover, based on Observation \ref{obs1.4}, it is natural to use biclique partitions to extend the results of Knuth \cite{Knuth} and Levine \cite{Levine} from line digraphs to general digraphs.

In Section 2, we give some basic definitions, notations and auxiliary lemmas. In Section 3, we give biclique partition formulas for counting oriented spanning trees and Eulerian circuits of digraphs, which generalize Theorems 1.1-1.3 to general digraphs. In Section 4, we give biclique partition formulas for stationary distribution vector and Kemeny's constant of digraphs. In Section 5, we give some concluding remarks, including more general spanning tree identity in digraphs, and the method for enumerating spanning trees of undirected graphs by vertex degrees and biclique partitions.

\section{Preliminaries}
Let $G$ be a weighted digraph on $n$ vertices, and each edge $e=ij\in E(G)$ is weighted by an indeterminate $w_{ij}(G)$. The \textit{weighted degree} of a vertex $i$ is $d_i(G)=\sum_{ij\in E(G)}w_{ij}(G)$. The \textit{Laplacian matrix} $L_G$ is the $n\times n$ matrix with entries
\begin{eqnarray*}
(L_G)_{ij}=\begin{cases}d_i(G)~~~~~~~~~~~~\mbox{if}~i=j,\\
-w_{ij}(G)~~~~~~~~\mbox{if}~ij\in E(G),\\
0~~~~~~~~~~~~~~~~~\mbox{otherwise}.\end{cases}
\end{eqnarray*}

Let $A(i,j)$ denote the submatrix of a matrix $A$ obtained by deleting the $i$-th row and the $j$-th column, and let $\det(A)$ denote the determinant of a square matrix $A$. The following lemma follows from the all minors matrix tree theorem \cite{Chaiken}.
\begin{lem}\label{lem2.1}
\textup{\cite{Chaiken}} Let $G$ be a weighted digraph with $n$ vertices. For any $i,j\in\{1,\ldots,n\}$, we have
\begin{eqnarray*}
\det(L_G(i,j))=(-1)^{i+j}t_i(G).
\end{eqnarray*}
\end{lem}
The following is a fomula for counting Eulerian circuits of digraphs.
\begin{lem}\label{lem2.2}
\textup{\cite{Aardenne,Smith}} Let $G$ be a Eulerian digraph. For any $u\in V(G)$, we have
\begin{eqnarray*}
\mathcal{E}(G)=\kappa_u(G)\prod_{v\in V(G)}(d_v^+(G)-1)!.
\end{eqnarray*}
\end{lem}

Let $G$ be a strongly connected weighted digraph with $n$ vertices, and all weights of $G$ are positive. The \textit{transition probability matrix} $P_G$ of $G$ is the $n\times n$ matrix with entries
\begin{eqnarray*}
(P_G)_{ij}=\begin{cases}w_{ij}(G)d_i(G)^{-1}~~~~~~~\mbox{if}~ij\in E(G),\\
0~~~~~~~~~~~~~~~~~~~~~~~~\mbox{if}~ij\notin E(G).\end{cases}
\end{eqnarray*}
A random walk on $G$ is defined by $P_G$, that is, $(P_G)_{ij}$ denotes the probability of moving from vertex $i$ to vertex $j$. Notice that $P_G$ is an irreducible nonnegative matrix with spectral radius $1$, and the all-ones vector is a right eigenvector for the eigenvalue $1$. By the Perron-Frobenius theorem, there exists a unique positive vector $\pi(G)=(\pi_1(G),\ldots,\pi_n(G))^\top$ such that $\pi(G)^\top P_G=\pi(G)^\top$ and $\sum_{i=1}^n\pi_i(G)=1$. Such vector $\pi(G)$ is called the \textit{stationary distribution vector} \cite{Beveridge} of $G$.
\begin{lem}\label{lem2.3}
\textup{[12, page 81]} Let $G$ be a strongly connected weighted digraph with $n$ vertices, and all weights of $G$ are positive. Then
\begin{eqnarray*}
\pi_i(G)=\frac{\det(I-P_G(i,i))}{\sum_{j=1}^n\det(I-P_G(j,j))},~i=1,\ldots,n.
\end{eqnarray*}
\end{lem}

Let $m_{ij}$ be the mean first passage time from vertex $i$ to vertex $j$, then the value $\mathcal{K}(G)=\sum_{j\in V(G),j\neq i}m_{ij}\pi_j(G)$ (does not depend on $i$) is called the \textit{Kemeny's constant} of $G$.
\begin{lem}\label{lem2.4}
\textup{[12, page 82]} Let $G$ be a strongly connected weighted digraph with positve weights, and $\lambda_1=1,\lambda_2,\ldots,\lambda_n$ are eigenvalues of $P_G$. Then
\begin{eqnarray*}
\mathcal{K}(G)=\sum_{i=2}^{n}\frac{1}{1-\lambda_i}.
\end{eqnarray*}
\end{lem}
For a matrix $E$, let $E[i_1,\ldots,i_s|j_1,\ldots,j_t]$ denote an $s\times t$ submatrix of $E$ whose row indices and column indices are $i_1,\ldots,i_s$ and $j_1,\ldots,j_t$, respectively. The following is a determinant identity involving the Schur complement.
\begin{lem}\label{lem2.5}
\textup{[19, Lemma 2.6]} Let $M=\begin{pmatrix}A&B\\C&D\end{pmatrix}$ be a block matrix of order $n$, where $A=M[1,\ldots,k|1,\ldots,k]$ is nonsingular. If $k+1\leq i_1<\cdots<i_s\leq n$ and $k+1\leq j_1<\cdots<j_s\leq n$, then
\begin{eqnarray*}
\frac{\det(M[1,\ldots,k,i_1,\ldots,i_s|1,\ldots,k,j_1,\ldots,j_s])}{\det(A)}=\det(S[i_1,\ldots,i_s|j_1,\ldots,j_s]),
\end{eqnarray*}
where $S=D-CA^{-1}B$ is the Schur complement of $D$ in $M$.
\end{lem}

\section{Oriented spanning trees of digraphs}
A biclique is a bipartite digraph $Q$ whose vertices can be partitioned into two parts $Q^{(1)}$ and $Q^{(2)}$, and $E(Q)=\{ij:i\in Q^{(1)},j\in Q^{(2)}\}$. Let $\varepsilon=\{Q_1,\ldots,Q_r\}$ be a biclique partition of a weighted digraph $G$ whose weights are induced by indeterminates $\{w_i(G)\}_{i\in V(G)}$. Let $\Omega(\varepsilon)$ denote the weighted biclique digraph with vertex set $\varepsilon=\{Q_1,\ldots,Q_r\}$ and edge set $\{Q_iQ_j:Q_i^{(2)}\cap Q_j^{(1)}\neq\emptyset\}$, and the weight of $Q_iQ_j$ in $\Omega(\varepsilon)$ is
\begin{equation}
w_{Q_iQ_j}(\Omega(\varepsilon))=w(Q_j)\sum_{u\in Q_i^{(2)}\cap Q_j^{(1)}}\frac{w_u(G)}{d_u(G)},\tag{3.1}
\end{equation}
where $w(Q_j)=\sum_{u\in Q_j^{(2)}}w_u(G)$. By Observation \ref{obs1.4}, we know that the part (1) of the following theorem extends Theorem \ref{thm1.3} to general digraphs.
\begin{thm}\label{thm3.1}
Let $\varepsilon=\{Q_1,\ldots,Q_r\}$ be a biclique partition of a weighted digraph $G$ whose weights are induced by indeterminates $\{w_i(G)\}_{i\in V(G)}$, and $d_u^+(G)>0$ for each $u\in V(G)$. Set $w(Q_i)=\sum_{u\in Q_i^{(2)}}w_u(G)$.\\
(1) For any $u\in V(G)$, we have
\begin{eqnarray*}
t_u(G)=\frac{w_u(G)\prod_{u\neq v\in V(G)}d_v(G)}{\prod_{i=1}^rw(Q_i)}\sum_{Q_i:u\in Q_i^{(2)}}t_{Q_i}(\Omega(\varepsilon)).
\end{eqnarray*}\\
(2) For any $Q_i\in\varepsilon$, we have
\begin{eqnarray*}
t_{Q_i}(\Omega(\varepsilon))&=&\frac{\prod_{i=1}^rw(Q_i)}{\prod_{u\in V(G)}d_u(G)}\sum_{u\in Q_i^{(1)}}t_u(G).
\end{eqnarray*}
\end{thm}
\begin{proof}
For a biclique partition $\varepsilon=\{Q_1,\ldots,Q_r\}$ in $G$, let $R_\varepsilon\in\mathbb{R}^{|V(G)|\times r}$ and $S_\varepsilon\in\mathbb{R}^{r\times|V(G)|}$ be two corresponding incidence matrices with entries
\begin{eqnarray*}
(R_\varepsilon)_{uj}&=&\begin{cases}w(Q_j)~~~~~~~~~~~\mbox{if}~u\in V(G),u\in Q_j^{(1)},\\
0~~~~~~~~~~~~~~~~~~\mbox{if}~u\in V(G),u\notin Q_j^{(1)},\end{cases}\\
(S_\varepsilon)_{iu}&=&\begin{cases}w_u(G)~~~~~~~~~~~\mbox{if}~u\in V(G),u\in Q_i^{(2)},\\
0~~~~~~~~~~~~~~~~~~\mbox{if}~u\in V(G),u\notin Q_i^{(2)},\end{cases}
\end{eqnarray*}
where $w(Q_j)=\sum_{u\in Q_j^{(2)}}w_u(G)$. Let $H$ be the bipartite weighted digraph with Laplacian matrix
\begin{eqnarray*}
L_H=\begin{pmatrix}D_1&-R_\varepsilon\\-S_\varepsilon&D_2\end{pmatrix},
\end{eqnarray*}
where $D_1$ is a diagonal matrix satisfying $(D_1)_{uu}=d_u(G)$, $D_2$ is a diagonal matrix satisfying $(D_2)_{ii}=w(Q_i)$. By computation, we have
\begin{eqnarray*}
D_1-R_\varepsilon D_2^{-1}S_\varepsilon=L_G,~D_2-S_\varepsilon D_1^{-1}R_\varepsilon=L_{\Omega(\varepsilon)}.
\end{eqnarray*}
For any $u\in V(G)$ and $Q_i\in\varepsilon$, by Lemmas \ref{lem2.1} and \ref{lem2.5}, we have
\begin{equation}
t_u(H)=\det(L_H(u,u))=t_u(G)\prod_{i=1}^rw(Q_i),\tag{3.2}
\end{equation}
\begin{equation}
t_{Q_i}(H)=\det(L_H(Q_i,Q_i))=t_{Q_i}(\Omega(\varepsilon))\prod_{u\in V(G)}d_u(G).\tag{3.3}
\end{equation}
Let $X$ be the adjoint matrix of $L_H$. Then
\begin{equation}
XL_H=\det(L_H)I=0.\tag{3.4}
\end{equation}
By Lemma \ref{lem2.1}, we get
\begin{equation}
(X)_{uv}=t_v(H)~~~~(u,v\in V(H)).\tag{3.5}
\end{equation}
Hence
\begin{eqnarray*}
\sum_{v\in V(H)}t_v(H)(L_H)_{vu}=d_u(G)t_u(H)-\sum_{i=1}^rt_{Q_i}(H)(S_\varepsilon)_{iu}=0.
\end{eqnarray*}
By (3.2) and (3.3) we get
\begin{eqnarray*}
d_u(G)t_u(G)\prod_{i=1}^rw(Q_i)-w_u(G)\prod_{v\in V(G)}d_v(G)\sum_{Q_i:u\in Q_i^{(2)}}t_{Q_i}(\Omega(\varepsilon))=0.
\end{eqnarray*}
Hence
\begin{eqnarray*}
t_u(G)=\frac{w_u(G)\prod_{u\neq v\in V(G)}d_v(G)}{\prod_{i=1}^rw(Q_i)}\sum_{Q_i:u\in Q_i^{(2)}}t_{Q_i}(\Omega(\varepsilon)).
\end{eqnarray*}
So part (1) holds.

By (3.4) and (3.5) we have
\begin{eqnarray*}
\sum_{v\in V(H)}t_v(H)(L_H)_{vQ_i}=w(Q_i)t_{Q_i}(H)-\sum_{u\in Q_i^{(1)}}t_u(H)(R_\varepsilon)_{ui}=0.
\end{eqnarray*}
By (3.2) and (3.3) we get
\begin{eqnarray*}
w(Q_i)t_{Q_i}(\Omega(\varepsilon))\prod_{u\in V(G)}d_u(G)-w(Q_i)\prod_{j=1}^rw(Q_j)\sum_{u\in Q_i^{(1)}}t_u(G)=0.
\end{eqnarray*}
Hence
\begin{eqnarray*}
t_{Q_i}(\Omega(\varepsilon))=\frac{\prod_{i=1}^rw(Q_i)}{\prod_{u\in V(G)}d_u(G)}\sum_{u\in Q_i^{(1)}}t_u(G),~i=1,\ldots,r.
\end{eqnarray*}
So part (2) holds.
\end{proof}

We can deduce the following result in \cite{Levine} from part (2) of Theorem \ref{thm3.1}.
\begin{cor}
\textup{[14, Theorem 1.1]} Let $G$ be a weighted digraph such that $d_u^-(G)>0$ for each $u\in V(G)$. Then
\begin{equation}
\sum_{e\in E(G)}t_e(\mathcal{L}(G))=\prod_{v\in V(G)}d_v(G)^{d_v^-(G)-1}\sum_{v\in V(G)}t_v(G),\tag{3.6}
\end{equation}
where $\mathcal{L}(G)$ is a weighted digraph whose weights are induced by indeterminates $\{w_e(G)\}_{e\in V(G)}$.
\end{cor}
\begin{proof}
If $d_u^+(G)=0$ for some vertex $u$ of $G$, then by $d_u^-(G)>0$, there exist two edges $e,f\in E(G)$ whose outdegrees are zeros in $\mathcal{L}(G)$. In this case, $t_e(\mathcal{L}(G))=0$ for each $e\in E(G)$, the left and right sides of (3.6) are both zeros.

If $d_u^+(G)>0$ for each vertex $u$ of $G$, then by Observation \ref{obs1.4} and part (2) of Theorem \ref{thm3.1}, we have
\begin{eqnarray*}
t_i(G)=\frac{\prod_{u\in V(G)}d_u(G)}{\prod_{u\in V(G)}d_u(G)^{d_u^-(G)}}\sum_{e\in E(G)\atop h(e)=i}t_e(\mathcal{L}(G)),\\
t_i(G)\prod_{v\in V(G)}d_v(G)^{d_v^-(G)-1}=\sum_{e\in E(G)\atop h(e)=i}t_e(\mathcal{L}(G)).
\end{eqnarray*}
Hence
\begin{eqnarray*}
\sum_{e\in E(G)}t_e(\mathcal{L}(G))=\sum_{i\in V(G)}\sum_{e\in E(G)\atop h(e)=i}t_e(\mathcal{L}(G))=\prod_{v\in V(G)}d_v(G)^{d_v^-(G)-1}\sum_{v\in V(G)}t_v(G).
\end{eqnarray*}
\end{proof}

For a biclique partition $\varepsilon=\{Q_1,\ldots,Q_r\}$ of digraph $G$, let $\Theta(\varepsilon)$ denote the weighted biclique digraph with vertex set $\varepsilon=\{Q_1,\ldots,Q_r\}$ and edge set $\{Q_iQ_j:Q_i^{(2)}\cap Q_j^{(1)}\neq\emptyset\}$, and the weight of $Q_iQ_j$ in $\Theta(\varepsilon)$ is
\begin{eqnarray*}
w_{Q_iQ_j}(\Theta(\varepsilon))=|Q_j^{(2)}|\sum_{u\in Q_i^{(2)}\cap Q_j^{(1)}}\frac{1}{d_u^+(G)}.
\end{eqnarray*}
Clearly, $\Theta(\varepsilon)$ is obtained from $\Omega(\varepsilon)$ by taking $w_v(G)=1$ for each $v\in V(G)$ in equation (3.1). We can deduce the following result from Theorem \ref{thm3.1}.
\begin{thm}\label{thm3.3}
Let $\varepsilon=\{Q_1,\ldots,Q_r\}$ be a biclique partition of a digraph $G$, and $d_u^+(G)>0$ for each $u\in V(G)$.\\
(1) For any $u\in V(G)$, we have
\begin{eqnarray*}
\kappa_u(G)=\frac{\prod_{u\neq v\in V(G)}d_v^+(G)}{\prod_{i=1}^r|Q_i^{(2)}|}\sum_{Q_i:u\in Q_i^{(2)}}t_{Q_i}(\Theta(\varepsilon)).
\end{eqnarray*}\\
(2) For any $Q_i\in\varepsilon$, we have
\begin{eqnarray*}
t_{Q_i}(\Theta(\varepsilon))&=&\frac{\prod_{i=1}^r|Q_i^{(2)}|}{\prod_{u\in V(G)}d_u^+(G)}\sum_{u\in Q_i^{(1)}}\kappa_u(G).
\end{eqnarray*}
\end{thm}

By Observation \ref{obs1.4}, we know that the following formulas extend Theorem \ref{thm1.2} to general Eulerian digraphs.
\begin{cor}\label{cor3.4}
Let $G$ be a Eulerian digraph. For any biclique partition $\varepsilon=\{Q_1,\ldots,Q_r\}$ of $G$, we have
\begin{eqnarray*}
\kappa_u(G)&=&\frac{\prod_{v\in V(G)}d_v^+(G)}{\prod_{j=1}^r|Q_j^{(2)}|}\frac{t_{Q_i}(\Theta(\varepsilon))}{|Q_i^{(1)}|},~i=1,\ldots,r,\\
\mathcal{E}(G)&=&\frac{\prod_{v\in V(G)}d_v^+(G)!}{\prod_{j=1}^r|Q_j^{(2)}|}\frac{t_{Q_i}(\Theta(\varepsilon))}{|Q_i^{(1)}|},~i=1,\ldots,r.
\end{eqnarray*}
\end{cor}
\begin{proof}
Since $G$ is Eulerian, we have $\kappa_u(G)=\kappa_v(G)$ for any $u,v\in V(G)$. By Theorem \ref{thm3.1}, we have
\begin{eqnarray*}
t_{Q_i}(\Theta(\varepsilon))&=&\frac{\prod_{j=1}^r|Q_j^{(2)}|}{\prod_{v\in V(G)}d_v^+(G)}\sum_{u\in Q_i^{(1)}}\kappa_u(G)=\frac{|Q_i^{(1)}|\prod_{j=1}^r|Q_j^{(2)}|}{\prod_{v\in V(G)}d_v^+(G)}\kappa_u(G),\\
\kappa_u(G)&=&\frac{\prod_{v\in V(G)}d_v^+(G)}{\prod_{j=1}^r|Q_j^{(2)}|}\frac{t_{Q_i}(\Theta(\varepsilon))}{|Q_i^{(1)}|}.
\end{eqnarray*}
By Lemma \ref{lem2.2}, we have
\begin{eqnarray*}
\mathcal{E}(G)=\kappa_u(G)\prod_{v\in V(G)}(d_v^+(G)-1)!=\frac{\prod_{v\in V(G)}d_v^+(G)!}{\prod_{j=1}^r|Q_j^{(2)}|}\frac{t_{Q_i}(\Theta(\varepsilon))}{|Q_i^{(1)}|}.
\end{eqnarray*}
\end{proof}

\section{Stationary distribution and Kemeny's constant of digraphs}
Let $\pi(G)$ denote the stationary distribution vector of a digraph $G$. By using biclique partitions of digraphs, we give the following reduction formulas for stationary distribution vector and Kemeny's constant of digraphs.
\begin{thm}\label{thm4.1}
Let $G$ be a strongly connected digraph with $n$ vertices. For any $u\in V(G)$ and biclique partition $\varepsilon=\{Q_1,\ldots,Q_r\}$ of $G$, we have
\begin{eqnarray*}
\pi_u(G)&=&\sum_{Q_i:u\in Q_i^{(2)}}|Q_i^{(2)}|^{-1}\pi_{Q_i}(\Theta(\varepsilon)),\\
\mathcal{K}(G)&=&\mathcal{K}(\Theta(\varepsilon))+n-r.
\end{eqnarray*}
\end{thm}
\begin{proof}
For a biclique partition $\varepsilon=\{Q_1,\ldots,Q_r\}$ in $G$, let $R_\varepsilon\in\mathbb{R}^{n\times r}$ and $S_\varepsilon\in\mathbb{R}^{r\times n}$ be two corresponding incidence matrices with entries
\begin{eqnarray*}
(R_\varepsilon)_{uj}&=&\begin{cases}d_u^+(G)^{-1}|Q_j^{(2)}|~~~~~~~\mbox{if}~u\in V(G),u\in Q_j^{(1)},\\
0~~~~~~~~~~~~~~~~~~~~~~~~\mbox{if}~u\in V(G),u\notin Q_j^{(1)}.\end{cases}\\
(S_\varepsilon)_{iu}&=&\begin{cases}|Q_i^{(2)}|^{-1}~~~~~~~~~~~~~~~\mbox{if}~u\in V(G),u\in Q_i^{(2)},\\
0~~~~~~~~~~~~~~~~~~~~~~~\mbox{if}~u\in V(G),u\notin Q_i^{(2)}.\end{cases}
\end{eqnarray*}
By computation, we have
\begin{eqnarray*}
R_\varepsilon S_\varepsilon=P_G,~S_\varepsilon R_\varepsilon=P_{\Theta(\varepsilon)}.
\end{eqnarray*}
Suppose that $P_{\Theta(\varepsilon)}=S_\varepsilon R_\varepsilon$ has eigenvalues $\lambda_1=1,\lambda_2,\ldots,\lambda_r$, then $P_G=R_\varepsilon S_\varepsilon$ has eigenvalues $\lambda_1,\lambda_2,\ldots,\lambda_r,0,\ldots,0$. By Lemma \ref{lem2.4}, we have
\begin{eqnarray*}
\mathcal{K}(G)=\frac{n-r}{1-0}+\sum_{i=2}^r\frac{1}{1-\lambda_i}=\mathcal{K}(\Theta(\varepsilon))+n-r.
\end{eqnarray*}
The products of all nonzero eigenvalues of $I-P_G$ and $I-P_{\Theta(\varepsilon)}$ are both
\begin{equation}
\prod_{i=2}^r(1-\lambda_i)=\sum_{v\in V(G)}\det(I-P_G(v,v))=\sum_{i=1}^r\det(I-P_{\Theta(\varepsilon)}(Q_i,Q_i)).\tag{4.1}
\end{equation}

Let $H$ be the bipartite weighted digraph with Laplacian matrix
\begin{eqnarray*}
L_H=\begin{pmatrix}I&-R_\varepsilon\\-S_\varepsilon&I\end{pmatrix}.
\end{eqnarray*}
For any $u\in V(G)$ and $Q_i\in\varepsilon$, by Lemmas \ref{lem2.1} and \ref{lem2.5}, we have
\begin{equation}
t_u(H)=\det(I-P_G(u,u)),~t_{Q_i}(H)=\det(I-P_{\Theta(\varepsilon)}(Q_i,Q_i)).\tag{4.2}
\end{equation}
Let $X$ be the adjoint matrix of $L_H$. Then
\begin{eqnarray*}
XL_H=\det(L_H)I=0.
\end{eqnarray*}
By Lemma \ref{lem2.1}, we get $(X)_{ij}=t_j(H)$. Hence
\begin{eqnarray*}
\sum_{v\in V(H)}t_v(H)(L_H)_{vu}=t_u(H)-\sum_{i=1}^rt_{Q_i}(H)(S_\varepsilon)_{iu}=0.
\end{eqnarray*}
By (4.2) we get
\begin{eqnarray*}
\det(I-P_G(u,u))=\sum_{Q_i:u\in Q_i^{(2)}}|Q_i^{(2)}|^{-1}\det(I-P_{\Theta(\varepsilon)}(Q_i,Q_i)).
\end{eqnarray*}
By Lemma \ref{lem2.3} and (4.1) we get
\begin{eqnarray*}
\pi_u(G)=\sum_{Q_i:u\in Q_i^{(2)}}|Q_i^{(2)}|^{-1}\pi_{Q_i}(\Theta(\varepsilon)).
\end{eqnarray*}
\end{proof}

\begin{example}
For a digraph $G$ with vertex set $V(G)=\{1,\ldots,n\}$, its $k$-blow up $G(k)$ has vertex set $V(G(k))=V_1\cup\cdots\cup V_n$ and edge set $E(G(k))=\bigcup_{ij\in E(G)}\{uv:u\in V_i,v\in V_j\}$, where $|V_1|=\cdots=|V_n|=k$. For a biclique partition $\varepsilon=\{Q_1,\ldots,Q_r\}$ of $G$, $\eta=\{P_1,\ldots,P_r\}$ is a biclique partition of $G(k)$, where $P_i$ is the $k$-blow up of $Q_i$. Notice that $w_{P_iP_j}(\Theta(\eta))=kw_{Q_iQ_j}(\Theta(\varepsilon))$ if $Q_i^{(2)}\cap Q_j^{(1)}\neq\emptyset$. By Theorem \ref{thm3.3}, we have
\begin{eqnarray*}
\kappa_u(G(k))=k^{nk-2}\prod_{v\in V(G)}d_v^+(G)^{k-1}\kappa_u(G)~(u\in V_i).
\end{eqnarray*}
By Theorem \ref{thm4.1}, we have
\begin{eqnarray*}
\pi_u(G(k))&=&k^{-1}\pi_i(G)~(u\in V_i),\\
\mathcal{K}(G(k))&=&\mathcal{K}(G)+n(k-1).
\end{eqnarray*}
\end{example}

By Observation \ref{obs1.4} and Theorem \ref{thm4.1}, we get the following formulas for stationary distribution vector and Kemeny's constant of line digraphs.
\begin{cor}\label{cor4.2}
Let $G$ be a strongly connected digraph with $n$ vertices and $m$ edges. For any $e=ij\in E(G)$, we have
\begin{eqnarray*}
\pi_e(\mathcal{L}(G))&=&d_i^+(G)^{-1}\pi_i(G),\\
\mathcal{K}(\mathcal{L}(G))&=&\mathcal{K}(G)+m-n.
\end{eqnarray*}
\end{cor}

Take $\mathcal{L}^0(G)=G$, and the iterated line digraph $\mathcal{L}^s(G)=\mathcal{L}(\mathcal{L}^{s-1}(G))$ ($s=1,2,3,\ldots$). We can get the following formula involving iterated line digraph from Corollary \ref{cor4.2}.
\begin{cor}
Let $G$ be a strongly connected digraph with $n$ vertices. Then
\begin{eqnarray*}
\mathcal{K}(\mathcal{L}^s(G))=\mathcal{K}(G)+|V(\mathcal{L}^s(G))|-n.
\end{eqnarray*}
\end{cor}

\section{Concluding remarks}
Let $G$ be a weighted digraph with a partition $V(G)=V_1\cup V_2$, and let $L_G=\begin{pmatrix}L_1&-B\\-C&L_2\end{pmatrix}$, where $L_1$ and $L_2$ are principal submatrices of $L_G$ corresponding to $V_1$ and $V_2$, respectively. If $L_1$ and $L_2$ are nonsingular, then $S_1=L_1-BL_2^{-1}C$ and $S_2=L_2-CL_1^{-1}B$ are Laplacian matrices of some weighted digraphs $G_1$ and $G_2$, respectively (because $S_1$ and $S_2$ are square matrices whose all row sums are zeros). Similar with the proof of Theorem \ref{thm3.1}, we can derive the following more general spanning tree identity
\begin{eqnarray*}
d_u(G)t_u(G_1)-\sum_{v\in V_1,vu\in E(G)}w_{vu}(G)t_v(G_1)=\frac{\det(L_1)}{\det(L_2)}\sum_{v\in V_2,vu\in E(G)}w_{vu}(G)t_v(G_2).
\end{eqnarray*}

We can also obtain new reduction formula for counting spanning trees in undirected graphs from our results. For a connected undirected graph $H$, let $H_0$ denote the digraph obtained from $H$ by replacing every edge $\{i,j\}\in E(H)$ with two directed edges $ij$ and $ji$. Then the number of spanning trees in $H$ is equal to $\kappa_u(H_0)$ for each $u\in V(H)$. For any biclique partition $\varepsilon=\{Q_1,\ldots,Q_r\}$ of $H_0$, by Corollary \ref{cor3.4}, we have
\begin{eqnarray*}
\kappa_u(H_0)=\frac{\prod_{v\in V(H)}d_v}{\prod_{j=1}^r|Q_j^{(2)}|}\frac{t_{Q_i}(\Omega(\varepsilon))}{|Q_i^{(1)}|},~i=1,\ldots,r,
\end{eqnarray*}
where $d_v$ is the degree of vertex $v$ in $H$.

\vspace{3mm}
\noindent
\textbf{Acknowledgements}

\vspace{3mm}

This work is supported by the National Natural Science Foundation of China (No. 12071097), and the Natural Science Foundation of the Heilongjiang Province (No. YQ2022A002).

\end{CJK*}
\end{spacing}
\end{document}